\documentclass{amsart}

\usepackage{amsmath, amssymb}
\input xy
\xyoption{all}
\begin{document}

\newtheorem{theorem}{Theorem}[section]
\newtheorem{lemma}[theorem]{Lemma}
\newtheorem{corollary}[theorem]{Corollary}
\newtheorem{fact}[theorem]{Fact}
\newtheorem{proposition}[theorem]{Proposition}
\newtheorem{claim}[theorem]{Claim}
\theoremstyle{definition}
\newtheorem{example}[theorem]{Example}
\newtheorem{remark}[theorem]{Remark}
\newtheorem{definition}[theorem]{Definition}
\newtheorem{question}[theorem]{Question}

\def\span{\operatorname{span}}
\def\cb{\operatorname{Cb}}
\def\tp{\operatorname{tp}}
\def\stp{\operatorname{stp}}
\def\acl{\operatorname{acl}}
\def\alg{\operatorname{alg}}
\def\dcl{\operatorname{dcl}}
\def\eq{\operatorname{eq}}
\def\loc{\operatorname{loc}_\Delta}
\def\C{\mathcal C}
\def\jet{\operatorname{Jet}}
\def\alg{\operatorname{alg}}
\def\rank{\operatorname{rank}}
\def\p{\bf P}
\def\PP{\mathbb{P}}
\def\UU{\mathbb{U}}
\def\aut{\operatorname{Aut}}
\def\dcf{\operatorname{DCF}_0}
\def\Gm{\mathbb G_{\operatorname m}}
\def\Ga{\mathbb G_{\operatorname a}}


\def\Ind#1#2{#1\setbox0=\hbox{$#1x$}\kern\wd0\hbox to 0pt{\hss$#1\mid$\hss}
\lower.9\ht0\hbox to 0pt{\hss$#1\smile$\hss}\kern\wd0}
\def\ind{\mathop{\mathpalette\Ind{}}}
\def\Notind#1#2{#1\setbox0=\hbox{$#1x$}\kern\wd0\hbox to 0pt{\mathchardef
\nn=12854\hss$#1\nn$\kern1.4\wd0\hss}\hbox to
0pt{\hss$#1\mid$\hss}\lower.9\ht0 \hbox to
0pt{\hss$#1\smile$\hss}\kern\wd0}
\def\nind{\mathop{\mathpalette\Notind{}}}

\title[Jet spaces preserve internality to the constants]{Differential-algebraic jet spaces preserve internality to the constants}
\author{Zoe Chatzidakis}
\address{Zoe Chatzidakis\\
\'Equipe de Logique Math\'ematique\\
Institut de Math\'ematiques de Jussieu - Paris Rive Gauche\\
Universit\'e Paris Diderot\\
UFR de mathŽmatiques case 7012\\
75205 Paris Cedex 13\\
France}
\email{zoe@math.univ-paris-diderot.fr}

\author{Matthew Harrison-Trainor}
\address{Matthew Harrison-Trainor\\
University of California, Berkeley\\
Department of Mathematics\\
970 Evans Hall\\
Berkeley, CA 94720-3840\\
USA}
\email{mattht@math.berkeley.edu}

\author{Rahim Moosa}
\address{Rahim Moosa\\
University of Waterloo\\
Department of Pure Mathematics\\
200 University Avenue West\\
Waterloo, Ontario \  N2L 3G1\\
Canada}
\email{rmoosa@uwaterloo.ca}

\thanks{Zo\'e Chatzidakis was partially supported by the MODIG grant ANR-09-BLAN-0047}

\thanks{Matthew Harrison-Trainor was partially supported by an NSERC USRA}

\thanks{Rahim Moosa was partially supported by an NSERC Discovery Grant as well as the MODIG grant ANR-09-BLAN-0047.
He would also like to thank the \'Equipe de Logique Math\'ematique of the Institut de Math\'ematiques de Jussieu at the Universit\'e Paris Diderot for their generous hospitality during his stay in May 2013 when some of the work presented here was completed}

\date{November 5th, 2013}

\begin{abstract}
Suppose $p$ is the generic type of a differential-algebraic jet space to a finite dimensional differential-algebraic variety at a generic point.
It is shown that $p$ satisfies a certain strengthening of almost internality to the constants.
This strengthening, which was originally called ``being Moishezon to the constants" in~\cite{moishezon} but is here renamed {\em preserving internality to the constants}, is a model-theoretic abstraction of the generic behaviour of jet spaces in complex-analytic geometry.
An example is given showing that only a generic analogue holds in the differential-algebraic case: there is a finite dimensional differential-algebraic variety $X$ with a subvariety $Z$ that is internal to the constants, such that the restriction of the differential-algebraic tangent bundle of $X$ to $Z$ is not almost internal to the constants.
\end{abstract}

\maketitle

\section{Introduction}

\noindent
This paper has to do with the fine structure of finite dimensional definable sets in differentially closed fields of characteristic zero.
A somewhat new and powerful tool in the study differential-algebraic varieties is the {\em differential jet space}.
This higher order analogue of Kolchin's differential tangent space was introduced by Pillay and Ziegler in~\cite{pillayziegler03} where it was used to prove what is now called the Canonical Base Property; a strong property which, among other things, gives a quick and Zariski-geometry-free proof of the Zilber dichotomy for differentially closed fields.
Here we study differential jet spaces in their own right, and prove that they satisfy a certain strengthening of internality to the constants introduced implicitly by the third author and Pillay in~\cite{cbp}, and then refined and formalised in~\cite{moishezon}.
This strengthening of internality is the differential analogue of a property that complex-analytic jet spaces enjoy, and went provisionally by the name ``being Moishezon" in~\cite{moishezon}.
However, in retrospect we find the term misleading and would like to rename it here as follows:

\begin{definition}
Work in a sufficiently saturated model $\overline M^{\operatorname{eq}}$ of a complete stable theory, and suppose ${\bf P}$  is an $\aut(\overline M)$-invariant set of partial types.
We say that a stationary type $\tp(a/b)$ {\em preserves ${\bf P}$-internality} if whenever $c$ is such that $\stp(b/c)$ is almost ${\bf P}$-internal, then so is $\stp(a/c)$.
\end{definition}

By taking $c=b$ we see that this is a strengthening of almost ${\bf P}$-internality.
See Proposition~2.4 of~\cite{moishezon} for a list of its basic properties.

Here is our main theorem.

\begin{theorem}
\label{jets-moishezon}
If $X$ is  a finite dimensional differential-algebraic variety then internality to the constant field is preserved by the generic types of the differential jet spaces to $X$ at generic points.

More concretely, working in a sufficiently saturated partial differentially closed field $(K,\Delta)$ of characteristic zero, suppose that $X$ is  a finite dimensional irreducible $\Delta$-variety defined over a $\Delta$-field $k$,  $a\in X$ is generic over $k$, and $v\in\jet^m_{\Delta}(X)_a$ is generic over $k\langle a\rangle$ for some $m>0$.
Then for any algebraically closed $\Delta$-field $L$ extending $k$, if $\tp(a/L)$ is almost internal to the constant field then so is $\tp(v/L)$.
\end{theorem}

A precise definition of the differential jet space $\jet^m_{\Delta}(X)_a$ is recalled in~$\S$\ref{prelims}, where we also prove something new that we need about differential jet spaces of fibred products (see Proposition~\ref{prod} below).

The theorem itself is proved in~$\S$\ref{proofsect}.

It may be worth translating this theorem into purely differential-algebraic terms.
First of all, note that for the generic type of an irreducible $\Delta$-variety $X$ to be almost internal to the constant field is equivalent to the following geometric property: there exists an irreducible algebraic variety $V$ over $\C$ and an irreducible $\Delta$-subvariety $\Gamma\subset X\times V(\C)$ that projects generically finite-to-one and $\Delta$-dominantly onto both $X$ and $V(\C)$.
In this case we will say that $X$ is {\em $\C$-algebraic}.
Theorem~\ref{jets-moishezon} can be restated as:

\begin{corollary}
\label{jets-moishezon-geometric}
Suppose $X$ is  a finite dimensional irreducible $\Delta$-variety and consider the $\Delta$-jet bundle $\jet^m_{\Delta}(X)\to X$, for any $m>0$.
Suppose $Y$ is a $\Delta$-variety and $Z\subset X\times Y$ a $\Delta$-subvariety  such that
\begin{itemize}
\item
$Z$ projects $\Delta$-dominantly onto both $X$ and $Y$, and
\item
for $c\in Y$ generic, $Z_c\subset X$ is irreducible and $\C$-algebraic.
\end{itemize}
Then, for generic $c\in Y$, the restriction of the $\Delta$-jet bundle to $Z_c$, $\jet^m_{\Delta}(X)|_{Z_c}$, is $\C$-algebraic.
\end{corollary}

One could ask for a more robust geometric statement; one could ask that for any $\C$-algebraic $\Delta$-subvariety $Z\subset X$, $\jet^m_{\Delta}(X)|_Z$ is again $\C$-algebraic.
Indeed, the analogous statement for complex-analytic jet spaces is true.
In~$\S$\ref{examplesect}, however, we will give a counterexample showing that this expectation does not hold (even in the case of $m=1$, so for $\Delta$-tangent bundles).
So while we view this work as furthering the analogy that model theory provides between differential-algebraic geometry and complex-analytic goemetry, that analogy is not perfect.



\medskip

Throughout this paper all our fields are of characteristic zero.

\bigskip
\section{Preliminaries on differential jet spaces}
\label{prelims}

\noindent
In this section we review the theory of differential jet spaces introduced by Pillay and Ziegler in~\cite{pillayziegler03}, and then prove something about how they interact with fibred products (Proposition~\ref{prod}).

We assume some familiarity with the theory of (partial) differentially closed fields, $\operatorname{DCF}_{0,\ell}$, as well as the associated differential-algebraic geometry, see for example~\cite{marker96}.
If it is more convenient, the reader is welcome to assume that $\ell=1$ as there is little difference between this and the general case when one is studying finite dimensional definable sets.
We work in a sufficiently saturated model $(K,\Delta)\models\operatorname{DCF}_{0,\ell}$ with field of total constants $\C$.
While everything can be made sense of in a more abstract setting, we will work in the strictly affine setting and identify geometric objects with their $K$-points.
So, for us a {\em $\Delta$-variety} is simply a Kolchin closed subset $X\subseteq K^n$.
To say that $X$ is {\em finite dimensional} is to say that if $k$ is a differential field over which $X$ is defined and $a\in X$, then the differential field $k\langle a\rangle _\Delta$ is of finite transcendence degree over $k$.
It is a fact that every finite dimensional $\Delta$-variety is $\Delta$-birationally equivalent to an ``algebraic $D$-variety" in the sense of Buium~\cite{buium-book}; that is, of the form
$$(V,s)^\sharp:=\{a\in V:(a,\partial_1a,\dots,\partial_\ell a)=s(a)\}$$
where $V$ is an algebraic variety and $s:V\to\tau V$ is an algebraic section to the prolongation of $V$.
(See, for example, 3.7(ii) of~\cite{pillayziegler03} for the ordinary case and~3.10 of~\cite{leonsanchez12} for the partial case.)

Pillay and Ziegler~\cite{pillayziegler03} introduced differential jet spaces for differential varieties of the form $X=(V,s)^\sharp$.
The {\em $m$th $\Delta$-jet space of $X$ at $a$} is a certain finite dimensional $\C$-vector subspace of the algebraic $m$th jet spaces of $V$ at $a$.
Let us first recall the algebraic notion: the algebraic $m$th jet space of the algebraic variety $V$ at a point $a\in V$ is by definition
$$\jet^m(V)_a=\hom_K(\mathcal M_{V,a}/\mathcal M_{V,a}^{m+1},K).$$
To give an explicit co-ordinate description of $\jet^m(V)_{a}$ as a definable $K$-vector space, fix an affine embedding $V\subseteq\mathbb A^{n}$ with co-ordinates $x=(x_{1},\dots,x_{n})$.
We can identify $\jet^m(\mathbb A^{n})_{a}=K^{\Lambda}$ where
$\Lambda:=\{\alpha\in\mathbb N^{n}:0<\sum_j\alpha_j\leq m\}$.
Then, using $z=(z_{\alpha})_{\alpha\in\Lambda}$ as co-ordinates for $K^{\Lambda}$, we have that $\jet^m(V)_{a}$ is the $K$-linear subspace of $K^{\Lambda}$ defined by the equations
\begin{eqnarray*}
\sum_{\alpha\in\Lambda}\frac{\partial^{\alpha}P}{\alpha ! \partial x^\alpha}(a)z_{\alpha}&=&0
\end{eqnarray*}
as $P$ ranges through a generating set for the ideal of $V$.
For details on this co-ordinate description for algebraic jet spaces see, for example, $\S2$ of~\cite{pillayziegler03} or~$\S5.1$ of~\cite{paperA}.

As is explained in~\cite{pillayziegler03}, if $a\in X=(V,s)^\sharp$ then $s$ induces a $\Delta$-module\footnote{Recall that a $\Delta$-module is a $K$-vector space $M$ equipped with additive endomorphisms $d=(d_1,\dots,d_\ell)$ satisfying $d_i(ra)=\partial_i(r)a+rd_i(a)$, for all $r\in K$ and $a\in M$.}
structure on $\mathcal M_{V,a}/\mathcal M_{V,a}^{m+1}$, say $d=(d_1,\dots, d_\ell)$, which in turn gives a $\Delta$-module structure to the dual space $\jet^m(V)_a$.
We denote this by $D=(D_1,\dots,D_\ell)$.
So for $\mu\in\mathcal M_{V,a}/\mathcal M_{V,a}^{m+1}$ and $v\in\jet^m(V)_a$, we have $(D_i v)(\mu):=\partial_i\big(v(\mu)\big)-v(d_i\mu)$.
The {\em $m$th $\Delta$-jet space of $X$ at $a$} is then defined to be the subspace
$$\jet^m_\Delta(X)_a:=\{v\in\jet^m(V)_a: Dv=0\}.$$
The construction is uniform in $a$, in the sense that if $\jet^m V\to V$ is the morphism of algebraic varieties whose fibres are the algebraic jet spaces, then we have a $\Delta$-subvariety $\jet^m_\Delta X\subseteq\jet^mV$ that maps onto $X$ and whose fibres are the $\Delta$-jet spaces.

The above construction was generalised to arbitrary (possibly infinite dimensional) differential subvarieties of algebraic varieties by the third author and Scanlon in~\cite{paperB}, where also various other theories of fields with operators were treated uniformly (the difference case was already developed by Pillay and Ziegler).
We do not give the general definition here, and only rely on~\cite{paperB} as a crutch to talk about $\jet^m_\Delta X$ even when $X$ is not given explicitly as the sharp points of an algebraic $D$-variety.
Of course, as we are only interested here in the finite dimensional case, we could always give such a presentation of $X$ after $\Delta$-birational change -- but it is convenient for us not to always insist on this.

Here are some basic facts about $\Delta$-jet spaces that can be easily deduced from the construction and can be found in~\cite{pillayziegler03} (though sometimes only implicitly).

\begin {fact}
\label{jetfacts}
Fix $m>0$.
\begin{itemize}
\item[(a)]
If $X$ is finite dimensional and $a\in X$ then $\jet^m_\Delta(X)_a$ is a finite dimensional $\C$-vector space.
\item[(b)]
Given a morphism of $\Delta$-varieties, $f:X\to Y$, there is a canonical morphism $\jet^m_\Delta f:\jet^m_\Delta X\to\jet^m_\Delta X$ such that the following commutes:
$$\xymatrix{
\jet^m_\Delta X\ar[d]\ar[rr]^{\jet^m_\Delta f}&&\jet^m_\Delta X\ar[d]\\
X\ar[rr]^{f} &&Y
}$$
and making $\jet^m_\Delta$ a covariant functor on the category of $\Delta$-varieties.
\item[(c)]
If $f:X\to Y$ is $\Delta$-dominant and $a\in X$ is generic, then $\jet^m_\Delta(f)_a:\jet^m_\Delta(X)_a\to\jet^m_\Delta(Y)_{f(a)}$ is a surjective $\C$-linear map.
If moreover, $f$ is generically finite-to-one, then $\jet^m_\Delta(f)_a:\jet^m_\Delta(X)_a\to\jet^m_\Delta(Y)_{f(a)}$ is an isomorphism.
\item[(d)]
If $X=V(\C)$ where $V$ is an algebraic variety over $\C$, then $\jet^m_\Delta X$ is the set of $\C$-points of the algebraic jet space $\jet^mV$.
\end{itemize}
\end{fact}

The following is a property of $\Delta$-jet spaces that does not seem to be covered in the literature.

\begin{proposition}
\label{prod}
Suppose that for $i=1,2$ we have $a_i\in X_i=(V_i,s_i)^\sharp$.
Then
$$\jet^m_\Delta(X_1\times X_2)_{(a_1,a_2)}\subseteq\dcl\big(\jet^m_\Delta(X_1)_{a_1}, \jet^m_\Delta(X_2)_{a_2},\C\big).$$
\end{proposition}

\begin{remark}
When $m=1$ the differential $m$th jet spaces is nothing other than Kolchin's differential tangent space from~\cite{kolchin85}.
In that case, Proposition~\ref{prod} is much easier to see as differential tangent spaces commute with products.
Except for the proof of this proposition (which becomes unnecessary), the proof of Theorem~\ref{jets-moishezon} does not change if we restrict our attention to the case of $m=1$ only, and the reader is therefore invited to do so if he or she prefers.
\end{remark}

\begin{proof}
We begin with some preliminary observations about $\Delta$-modules.

First, recall from~\cite{pillayziegler03} that if $(M,d)$ is a  $\Delta$-module then we obtain a dual $\Delta$-module, $\big(\hom_K(M,K),D\big)$, by defining $D_iv:M\to K$ by $a\mapsto\partial_i(v(a))-v(d_ia)$.
Let us, somewhat unusually, set
$$M^\Delta:=\{v\in\hom_K(M,K):Dv=0\}$$
Then~3.1 of~\cite{pillayziegler03} tells us that $M^\Delta$ is a $\C$-vector space of dimension $\dim_KM$, when the latter is finite.

We can also take tensor products of $\Delta$-modules.
If $(M,d_M)$ and $(N,d_N)$  are  $\Delta$-modules then it is not hard to verify that we get a $\Delta$-module $(M\otimes_KN,d)$ by $d(a\otimes b)=d_Ma\otimes b+a\otimes d_Nb$.

Fix finite $K$-dimensional $\Delta$-modules $(M,d_M)$ and $(N,d_N)$ and consider the $K$-bilinear map
$$\phi:\hom_K(M,K)\times\hom_K(N,K)\to \hom_K(M\otimes_KN,K)$$
given by $\phi(v,w)(a\otimes b)=v(a)w(b)$.
Note that this is the natural map in linear algebra that induces an isomorphism between the tensor product of duals with the dual of a tensor product.
\begin{claim}
\label{map}
$\phi$ restricted to $M^\Delta\times N^\Delta$ induces an isomorphism of $\C$-vector spaces, $M^\Delta\otimes_\C N^\Delta\approx (M\otimes_KN)^\Delta$.
In particular, $\span_{\C}\big(\phi(M^\Delta\times N^\Delta)\big)=(M\otimes_KN)^\Delta$.

\end{claim}
\begin{proof}[Proof of Claim~\ref{map}]
First, if $v\in M^\Delta$ and $w\in N^\Delta$, then the following computation shows that $\phi(v,w)\in(M\otimes_KN)^\Delta$,
\begin{eqnarray*}
D\big(\phi(v,w)\big)(a\otimes b)
&=&
\partial\big(\phi(v,w)(a\otimes b)\big)-\phi(v,w)(d(a\otimes b))\\
&=&
\partial\big(v(a)w(b)\big)-\phi(v,w)\big(d_Ma\otimes b+a\otimes d_Nb\big)\\
&=&
(\partial v(a))w(b)+v(a)(\partial w(b))-v(d_Ma)w(b)-v(a)w(d_Nb)\\
&=&
\big(\partial v(a)-v(d_Ma)\big)w(b)+v(a)\big(\partial w(b)-w(d_Nb)\big)\\
&=&
(D_Mv(a))w(b)+ v(a)(D_Nw(b))\\
&=&
0w(b)+ v(a)0\\
&=&0
\end{eqnarray*}
Hence by $\C$-bilinearity we do get an induced $\C$-linear map
$$M^\Delta\otimes_\C N^\Delta\to(M\otimes_KN)^\Delta$$
Injectivity follows exactly as it does for the injectivity of the $K$-linear map induced by $\phi$.
As both have the same dimension, this induced map is an isomorphism.
\end{proof}

We are going to apply this claim to the $\Delta$-modules $M=M_1:=\mathcal O_{V_1,a_1}/\mathcal M_{V_1,a_1}^{m+1}$ and $N=M_2:=\mathcal O_{V_2,a_2}/\mathcal M_{V_2,a_2}^{m+1}$.
Note that $M_i=K\oplus\mathcal M_{V_i,a_i}/\mathcal M_{V_i,a_i}^{m+1}$ so that canonically
$$\hom_K(M_i,K)=K\oplus \jet^m(V_i)_{a_i}$$
where the direct sum is also in the sense of $\Delta$-modules if we put on $K$ the $\Delta$-module structure given by $\partial$.
Hence, taking constants, we get
\begin{equation}
\label{midelta}
M_i^{\Delta}=\C\oplus \jet_\Delta^m(X_i)_{a_i}
\end{equation}
as $\C$-vector spaces.

To represent things in co-ordinates, let us fix affine embeddings $V_i\subseteq\mathbb A^{n_i}$, set  $n:=n_1+n_2$, $z:=(x_1,\dots,x_{n_1},y_1,\dots,y_{n_2})$ co-ordinates for $\mathbb A^n$, and
$$a:=(a_1,a_2)\in X_1\times X_2\subset V_1\times V_2\subseteq\mathbb A^n$$
Consider the standard ``monomial" basis for $\mathcal M_{\mathbb A^n,a}/\mathcal M_{\mathbb A^n,a}^{m+1}$, 
$\{(z-a)^\alpha:\alpha\in\Lambda\}$,
where $\Lambda:=\{\alpha\in\mathbb N^n:0<\sum_j\alpha_j\leq m\}$.
Denote by $\overline{(z-a)}^\alpha$ their images in $\mathcal M_{V_1\times V_2,a}/\mathcal M_{V_1\times V_2,a}^{m+1}$.
There is a natural embedding
\begin{eqnarray*}
\mathcal M_{V_1\times V_2,a}/\mathcal M_{V_1\times V_2,a}^{m+1} &\subset& M_1\otimes_KM_2
\end{eqnarray*}
induced by writing $\overline{(z-a)}^\alpha=\overline{(x-a_1)}^{\alpha_1} \overline{(y-a_2)}^{\alpha_2}$.
Note that $\alpha_i$ may be zero for one of $i=1,2$, and this is why we work with $M_i=\mathcal O_{V_i,a_i}/\mathcal M_{V_i,a_i}^{m+1}$ rather than with $\mathcal M_{V_i,a_i}/\mathcal M_{V_i,a_i}^{m+1}$.
Now, we obtain a corresponding embedding of the dual spaces
\begin{eqnarray*}
\jet^m(V_1\times V_2)_a &\subset &\hom_K(M_1\otimes_KM_2,K)
\end{eqnarray*}
by extending $K$-linear functionals on $\mathcal M_{V_1\times V_2,a}/\mathcal M_{V_1\times V_2,a}^{m+1}$ to $M_1\otimes_K M_2$ by setting them to be zero where they were not defined. 
The above emvedding is also as $\Delta$-modules.
So
\begin{eqnarray*}
\jet_\Delta^m(X_1\times X_2)_a
&\subset&
(M_1\otimes_KM_2)^\Delta
\end{eqnarray*}
Putting this together with Claim~\ref{map} as well as~(\ref{midelta}), implies
\begin{equation}
\label{jetproduct}
\jet_\Delta^m(X_1\times X_2)_a \  \subset \  \big(\C\oplus \jet_\Delta^m(X_1)_{a_1}\big)\otimes_{\C}\big(\C\oplus \jet_\Delta^m(X_2)_{a_2}\big)
\end{equation}
It remains to trace through the various identifications to verify in co-ordinates that~(\ref{jetproduct}) does in fact lead to a proof of Proposition~\ref{prod}.

Fix $v\in \jet_\Delta^m(X_1\times X_2)_a$.
As $\jet_\Delta^m(X_1\times X_2)_a$ is a $\C$-linear subspace of $K^{\Lambda}$ we can thus write $v=(v_\alpha)_{\alpha\in\Lambda}$.
Note that viewing $v$ as linear function on $\mathcal M_{V_1\times V_2,a}/\mathcal M_{V_1\times V_2,a}^{m+1}$ we can compute
$$v_\alpha=v\big(\overline{(z-a)}^\alpha\big)$$
Now fix $\C$-bases $W$ and $W'$ for $\jet_\Delta^m(X_1)_{a_1}$ and $\jet_\Delta^m(X_2)_{a_2}$ respectively.
By~(\ref{jetproduct}) we get
$$v=c_1(1\otimes 1)+\sum_{w\in W}c_w(w\otimes 1)+\sum_{w'\in W'}c_{w'}(1\otimes w')+\sum_{w\in W,w'\in W'}c_{w,w'}(w\otimes w')$$
where the $c$'s are constants.
Evaluating both sides at $\overline{(z-a)}^\alpha=\overline{(x-a_1)}^{\alpha_1} \overline{(y-a_2)}^{\alpha_2}$, we get
\begin{equation}
\label{coord}
v_\alpha=c_11_{\alpha_1}1_{\alpha_2}+\sum_{w\in W}c_ww_{\alpha_1}1_{\alpha_2}+\sum_{w'\in W'}c_{w'}1_{\alpha_1}w'_{\alpha_2}+\sum_{w\in W,w'\in W'}c_{w,w'}w_{\alpha_1}w'_{\alpha_2}
\end{equation}
where $1_\beta=
\begin{cases}
1& \text{ if } \beta=0\\
0& \text{else}
\end{cases}$.
Equation~(\ref{coord}) shows explicitly in co-ordinates that $v\in\dcl(W,W',\C)$.
Hence $\jet^m_\Delta(X_1\times X_2)_a\subseteq\dcl\big(\jet^m_\Delta(X_1)_{a_1}, \jet^m_\Delta(X_2)_{a_2},\C\big)$, as desired.
\end{proof}

\bigskip
\section{The Proof of Theorem~\ref{jets-moishezon}}
\label{proofsect}

\noindent
We continue to work in a sufficiently saturated models $(K,\Delta)\models\operatorname{DCF}_{0,\ell}$ with field of total constants $\C$.
We begin with some minor reductions and notational simplifications.
First, fix $m>0$ and abbreviate $\jet^m_\Delta$ by $T$.
This will also serve to remind the reader that not much is lost if one considers simply the $\Delta$-tangent spaces, that is the case when $m=1$.
What we will use about $T$, freely and more or less axiomatically, are the facts stated and/or proved in the previous section.
Second, by working over $k$, we can drop all reference to this base field altogether.
Third, it clearly suffices to prove the theorem when $L$ is the algebraic closure of a finitely generated differential field (over $k$).
Hence, what we actually need to prove is that if $a$ is of finite dimension and $\stp(a/b)$ is almost $\C$-internal, then so is $\stp(v/b)$ for $v$ a generic point in $T\big(\loc(a)\big)_a$.

Our next reduction is to the case that $b\in\acl(a)$.
In fact this too is for convenience, in the sense that it is not essential to the proof.
However, for nontrivial (but known) reasons, we can actually reduce to this case: the first author has shown (Lemma~2.3 of~\cite{chatzidakis12}) that as a consequence of the Canonical Base Property, if $\stp(a/b)$ and $\stp(a/b')$ are almost $\C$-internal then so is $\tp\big(a/\acl(b)\cap\acl(b')\big)$.
Taking $b'=a$, we get that $\tp\big(a/\acl(a)\cap\acl(b)\big)$is almost $\C$-internal.
If we show that $\tp\big(v/\acl(a)\cap\acl(b)\big)$ is almost $\C$-internal then we get {\em a fortiori} that $\stp(v/b)$ is too.
That the Canonical Base Property in the required form holds in $\operatorname{DCF}_0$ was done by Pillay and Ziegler in~\cite{pillayziegler03}, and their argument was shown to extend to finite dimensional types in $\operatorname{DCF}_{0,\ell}$ by Leon Sanchez~\cite{omarthesis}.

We have thus reduced to showing the following statement:
\begin{itemize}
\item[$(*)$]
Suppose $a$ is a tuple of finite dimension, $b\in\acl(a)$, and $\stp(a/b)$ is almost $\C$-internal.
If $v$ is a generic point of $T\big(\loc(a)\big)_a$, then $\stp(v/b)$ is almost $\C$-internal.
\end{itemize}
We will proceed via a series of lemmas.

\begin{lemma}
\label{anyvector}
The statement~$(*)$ is equivalent to the version where the conclusion is made about all $v\in T\big(\loc(a)\big)_a$ rather than just the generic $v$.
\end{lemma}
\begin{proof}
Indeed, this is because $T\big(\loc(a)\big)_a$ is an $a$-definable (additive) group and hence every element is a sum of generics.
So $v\in\dcl(v_1v_2a)$ for a pair of generic points $v_1$ and $v_2$, and the almost $\C$-internality of $\stp(v/b)$ follows from that of $\stp(v_1/b), \stp(v_2/b),$ and $\stp(a/b)$.
\end{proof}

In what follows, whenever we say that ``$(*)$ holds for $(a,b)$" we mean that both the hypotheses and the conslusions hold.
In particular, $a$ is a tuple of finite dimension and $b\in\acl(a)$.

\begin{lemma}
\label{image}
Suppose $(*)$ holds for $(a,b)$.
If $\acl(b)\subseteq\acl(a')\subseteq\acl(a)$ then $(*)$ holds for $(a', b)$ as well.
In particular, $(*)$ is preserved if one replaces $a$ by anything interalgebraic with it.
\end{lemma}

\begin{proof}
Note that the hypotheses of $(*)$ hold automatically for $(a,b)$.
Since $a'\in\acl(a)$ there exists a $\Delta$-subvariety
$$Z\subseteq \loc(a')\times\loc(a)$$
which projects $\Delta$-dominantly onto both co-ordinates and is generically finite-to-one onto $\loc(a)$.
It follows that $T(Z)_{a'a}$ is $a'a$-definably isomorphic to $T\big(\loc(a)\big)_a$ and admits a surjective $a'a$-definable  map onto $T\big(\loc(a')\big)_{a'}$.
Hence if $v$ is generic in $T\big(\loc(a)\big)_a$ then its image $v'$ in $T\big(\loc(a')\big)_{a'}$ is generic, and the almost $\C$-internality of $\stp(v/b)$ implies that $\stp(v'/b)$ is also almost $\C$-internal.
\end{proof}

\begin{lemma}
\label{fibreproduct}
Suppose $(*)$ holds for $(a_1,b)$ and for $(a_2,b)$, and $a_1\ind_ba_2$.
Then $(*)$ holds for $(a_1a_2,b)$.
\end{lemma}

\begin{proof}
Actually, we will show $(*)$ for $(ba_1a_2,b)$ which is equivalent by Lemma~\ref{image}. 
Let $X=\loc(ba_1a_2)$, $Y_i=\loc(ba_i)$ for $=1,2$, and $B=\loc(b)$.
So $X=Y_1\times_BY_2\subseteq Y_1\times Y_2$.
Suppose for the moment that the $Y_i$ are in fact algebraic $D$-varieties so that we can apply Proposition~\ref{prod} directly to them.
Then
\begin{eqnarray*}
T(X)_{(ba_1a_2)} &\subseteq& T(Y_1\times Y_2)_{(ba_1,ba_2)}\\
&\subseteq&
\dcl\big(T(Y_1)_{ba_1}, T(Y_2)_{ba_2},\C\big) \ \ \text{ by Proposition~\ref{prod}}
\end{eqnarray*}
The truth of $(*)$ for $(a_i,b)$ implies by Lemma~\ref{image} the truth of $(*)$ for $(ba_i,b)$, and hence by Lemma~\ref{anyvector}, the type of every element of $T(Y_i)_{ba_i}$ over $b$ is almost $\C$-internal.
It follows that in particular the generic type of $T(X)_{(ba_1a_2)}$ is almost $\C$-internal, as desired.

It remains to verify that we may assume the $Y_i$ are algebraic $D$-varieties, that is that they are of the form $Y_i=(V_i,s_i)^\sharp$.
By finite dimensionality there exist algebraic $D$-varieties $\widehat Y_i$ admitting morphisms to $B$ such that $X=Y_1\times_BY_2$ is $\Delta$-birationally equivalent to $\widehat Y_1\times_B\widehat Y_2$.
Since we are trying to prove something about the $\Delta$-jet space to $X$ at a generic point, it suffices to prove that statement for $\widehat Y_1\times_{B}\widehat Y_2$, instead.
\end{proof}

We now prove an important case of~$(*)$.

\begin{lemma}
\label{orthogonal}
Suppose $\stp(a/b)$ is almost orthogonal to $\C$. Then $(*)$ holds for $(a,b)$.
\end{lemma}

\begin{proof}
As $\tp(a/b)$ is almost $\C$-internal, for some $e$ extending $b$, with
\begin{eqnarray}
\label{2}
a\ind_{b}e
\end{eqnarray}
and for some finite tuple of constants $c$, we have $a\in\acl(ec)$.
We may assume that $c$ is an algebraically independent tuple over $\acl(e)$, and hence
\begin{eqnarray}
\label{1}
c\ind e
\end{eqnarray}
In particular, $c\ind b$.
But by the almost orthogonality of $\stp(a/b)$ to $\C$, $a\ind_bc$.
This implies
\begin{eqnarray}
\label{3}
a\ind c
\end{eqnarray}

Now, let $X=\loc(a)$ and $Y=\loc(e)$.
Choose a (finite) $\C$-basis $\beta$ for $T(Y)_e$.
We may assume that $\beta\ind_e a$, and hence by~(\ref{2}),
\begin{eqnarray}
\label{4}
\beta e\ind_ba
\end{eqnarray}

Let $Z=\loc(ae/c)\subseteq X\times Y$.
By~(\ref{1}), $Y=\loc(e/c)$ and hence the projection $\pi_{Y}:Z\to Y$ is $\Delta$-dominant.
The fact that $a\in\acl(ec)$ implies that $\pi_{Y}$ is moreover generically finite-to-one.
By Fact~\ref{jetfacts}(c), it therefore induces a $cae$-definable linear isomorphism between $T(Z)_{ae}$ and $T(Y)_e$.
On the other hand, by~(\ref{3}), $X=\loc(a/ c)$, and so the projection $\pi_X:Z\to X$ is also $\Delta$-dominant, and we obtain a $cae$-definable surjective linear map from $T(Z)_{ae}$ to $T(X)_a$.
Putting these together we get a $cae$-definable surjective linear map from $T(Y)_e$ to $T(X)_a$.

To show $(*)$, take $v\in T(X)_a$ generic.
We may assume that $v\ind_{a}\beta e$.
So~(\ref{4}) implies
\begin{eqnarray*}
\label{5}
v\ind_b\beta e
\end{eqnarray*}
On the other hand,
\begin{eqnarray*}
v
&\in&
\dcl(cae,T(Y)_e) \ \ \ \  \text{by }T(Y)_e\to T(X)_a\\
&\subseteq&
\dcl(cae\beta\C) \ \ \ \ \text{as $\beta$ is a $\C$-basis for $T(Y)_e$}\\
&\subseteq&
\acl(e\beta\C) \ \ \ \ \text{as $a\in\acl(ec)$ and $c$ is from $\C$}
\end{eqnarray*}
That is, $\stp(v/b)$ is almost $\C$-internal, as desired.
\end{proof}

Our final lemma has nothing to do with jet spaces, and is simply a refinement of how internality to $\C$ can be witnessed in $\operatorname{DCF}_{0,\ell}$.

\begin{lemma}
\label{goode}
Suppose $a$ and $b$ are tuples of finite rank with $\stp(a/b)$ almost internal to $\C$.
Then there exist a tuple $e$ of finite rank and a tuple of constants $c$, such that
\begin{itemize}
\item[(i)]
$a\ind_b e$,
\item[(ii)]
$\acl(abe)=\acl(cbe)$,
\item[(iii)]
$\stp(e/b)$ is almost $\C$-internal and almost orthogonal to $\C$, and
\item[(iv)]
$c\ind be$
\end{itemize}
\end{lemma}

\begin{proof}
Almost $\C$-internality of $\stp(a/b)$ gives us an $e_1$ such that $a\ind_b e_1$ and $a$ is interalgebraic with a tuple of constants over $be_1$.
So there is a $be_1$-definable generically finite-to-finite correspondence, say $f_{e_1}$, between $\loc(a/b)$ and the $\C$-points of an algebraic variety over $\C$.
Let $M$ be a prime model over $\acl(b)$, independent from $a$ over $b$.
Then we can find  a tuple $e'$ in $M$ such that $f_{e'}$ is a $be'$-definable generically finite-to-finite correspondence between $\loc(a/b)$ and the $\C$-points of an algebraic variety over $\C$.
So there is a tuple of constants $c'$ such that $(a,b,e',c')$ satisfy (i) and (ii).
Since $M$ is prime over $\acl(b)$ it adds no new constants to $\acl(b)$, and hence, $stp(e'/b)$ is almost orthogonal to $\C$.
Setting $e:=\cb(\stp(ac'/be'))$, everything so far is preserved, but now $\stp(e/b)$ is finite rank and almost $\C$-internal (as that is the case for $a$ and $c'$).
So $(a,b,c',e)$ satisfy~(i)--(iii).
Letting $c$ be a transcendence basis for $\acl(c')$ over $be$, we get~(iv) as well. 
\end{proof}

We can now put things together.

\begin{proof}[Proof of~$(*)$ in general]
We are given $b\in\acl(a)$ such that $\stp(a/b)$ is almost $\C$-internal.
Let $e$ and $c$ be as given by Lemma~\ref{goode}.

First, by~(iii) of~\ref{goode}, $\stp(e/b)$ is almost $\C$-internal and almost orthogonal to $\C$.
Hence~$(*)$ holds of $(be,b)$ by Lemma~\ref{orthogonal}.

Next we observe that $(*)$ holds for $(cb,b)$.
Note that  $b\notin\acl(c)$, so it does not make sense to ask whether~$(*)$ holds for $(c,b)$.
We let $v$ be generic in $T\big(\loc(cb)\big)_{cb}$ and we show directly that $\stp(v/b)$ is $\C$-internal.
By~(iv) of~\ref{goode},  $c\ind b$, so that $\loc(cb)=\loc(c)\times\loc(b)$.
Hence
\begin{eqnarray*}
T\big(\loc(c,b)\big)_{(c,b)} &=& T\big(\loc(c)\times \loc(b)\big)_{(c,b)}\\
&\subseteq&
\dcl\big(T\big(\loc(c)\big)_{c}, T\big(\loc(b)\big)_{b},\C\big) \ \ \text{ by Proposition~\ref{prod}}
\end{eqnarray*}
To be precise, in order to apply Proposition~\ref{prod}, we need to first make $\Delta$-birational changes so that the $\Delta$-varieties in question are the sharp points of algebraic $D$-varieties, but this can be done as explained in the proof of Lemma~\ref{fibreproduct}.
Now,  as $c$ is a tuple of constants, $\loc(c)=V(\C)$ where $V$ is an algebraic variety over $\C$, and hence, by Fact~\ref{jetfacts}(d),  $T\big(\loc(c)\big)_{c}$ is the set of $\C$-point of the algebraic jet space of $V$ at $c$.
So every element in $T\big(\loc(c)\big)_{c}$ is a tuple of constants.
We thus have
$$T\big(\loc(c,b)\big)_{(c,b)}\subseteq\dcl\big(T\big(\loc(b)\big)_{b},\C\big)$$
But differential jet spaces at a point are finite dimensional $\C$-vector spaces, so choosing a $\C$-basis $\beta$ for $T\big(\loc(b)\big)_{b}$, we have
$$T\big(\loc(c,b)\big)_{(c,b)}\subseteq\dcl(\beta,\C)$$
As $\beta$ can be chosen to be independent of $v$ over $b$, we have shown that $\stp(v/b)$ is $\C$-internal, as desired.
So~$(*)$ holds for $(cb,b)$.

Now, by Lemma~\ref{goode}(iv), $c\ind_be$.
Therefore putting the previous two cases together using Lemma~\ref{fibreproduct}, we get that~$(*)$ holds for $(cbe,b)$.
But $\acl(cbe)=\acl(ae)$ by~(i) of~\ref{goode}, so Lemma~\ref{image} implies that~$(*)$ holds of $(ae,b)$.
Finally, by Lemma~\ref{image} again, $(*)$ is then true of $(a,b)$, as desired.
\end{proof}

\bigskip
\section{A Counterexample}
\label{examplesect}

\noindent
Recall from the Introduction that {\em $\C$-algebraic} means being in generically finite-to-finite correspondence with the constant points of an algebraic variety, or equivalently that the generic type is almost $\C$-internal.
To paraphrase Corollary~\ref{jets-moishezon-geometric}, we have shown that if $X$ is a finite dimensional $\Delta$-variety that is covered by a family of $\C$-algebraic subvarieties, then the restriction of the $\Delta$-jet bundle of $X$ to a generic member of this family is again $\C$-algebraic.
One might expect something more robust: Is the restriction of the $\Delta$-jet bundle of $X$ to {\em any} $\C$-algebraic subvariety again $\C$-algebraic?
Indeed, our experience with the model theory of compact complex manifolds suggests that the answer should be ``yes":
It follows from GAGA that if $X$ is a compact complex space, $TX\to X$ is its tangent bundle, and $Z\subset X$ is a Moishezon\footnote{Being {\em Moishezon} means being bimeromorphic to a projective algebraic variety.
This may seem too strong to be an analogue of $\C$-algebraicity, but it turns out that a compact complex space admitting a generically finite-to-finite correspondence with a projective algebraic variety is Moishezon.}
 subvariety, then the restriction $(TX)|_Z$ is again Moishezon.
A similar statement holds true of higher order complex-analytic jet spaces also.\footnote{There are various closely related notions of ``jet space" in complex geometry, see~\cite{moosa} for an exposition of them.}
Somewhat surprisingly, therefore, this strengthening of Corollary~\ref{jets-moishezon-geometric} does not hold in differential-algebraic geometry:

\begin{proposition}
\label{exampleprop}
There exists an irreducible finite dimensional $\Delta$-variety $X$ defined over $\C$ such that the restriction of $T_\Delta X$ to $X(\C)$ is not $\C$-algebraic.
\end{proposition}

Here, and throughout this section, $T_\Delta$ denotes Kolchin's $\Delta$-tangent bundle from~\cite{kolchin85}, which agrees with the first $\Delta$-jet bundle as defined in $\S$\ref{prelims} above.

Our example requires only a single derivation, and so to prove Proposition~\ref{exampleprop} we work in a sufficiently saturated model $(K,\delta)\models\operatorname{DCF}_0$ with constant field $\C$.

We will use the following fact that we expect is well known but for which we were unable to find a reference in the literature.

\begin{lemma}
\label{log}
Let $G\leq\mathbb G_{\operatorname m}$ be the subgroup defined by $\displaystyle \delta\left(\frac{\delta x}{x}\right)=0$.
Then $G$ is not $\C$-algebraic.
\end{lemma}

\begin{proof}
Let us denote by $\ell(x):=\frac{\delta x}{x}$ the logarithmic derivative operator, which is a group homomorphism from $\Gm$ to $\Ga$.
Restricting $\ell$ to $G$ we have the short exact sequence of definable group homomorphisms
$$\xymatrix{1\ar[r] &\Gm(\C)\ar[r] &G\ar[r]^{\ell \ \ \ \ } &\Ga(\C)\ar[r]&0}$$
Now, a finite dimensional definable group is $\C$-algebraic if and only if it is definably group isomorphic to the $\C$-points of an algebraic group over $\C$ -- see Corollary~3.10 of~\cite{pillay06}.
We may therefore assume, toward a contradiction, that $G$ is definably isomorphic to $H(\C)$ for some algebraic group $H$ over $\C$.
By the fact that $\C$ is stably embedded, the above exact sequence shows that $H$ is an extension of $\Ga$ by $\Gm$.
But by the structure of commutative linear algebraic groups, all such extensions are split, and so $H=\Gm\times\Ga$.
We therefore have a definable isomorphism
$$\xymatrix{
\Gm(\C)\times\Ga(\C)\ar[rr]^{ \ \ \ \ \approx}\ar[dr]&&G\ar[dl]^{\ell}\\
&\Ga(\C)
}$$
The upshot is that there is a $\delta$-field $L$, a generic point $(c,d)\in\Gm(\C)\times\Ga(\C)$ over $L$, and a rational function $f(x,y)$ over $L$, such that $a:=f(c,d)\in G$, and $\ell(a)=d$.

Let us write $f=\frac{P}{Q}$ where $P,Q\in L[x,y]$.
So we have
$$Q(c,d)a=P(c,d)$$
Applying $\ell$ to both sides, and using that $\ell(a)=d$, we get
$$\ell\big(Q(c,d)\big)+d=\ell\big(P(c,d)\big)$$
But as $c$ and $d$ are constant tuples, this becomes
$$\frac{Q^{\delta}(c,d)}{Q(c,d)}+d=\frac{P^{\delta}(c,d)}{P(c,d)}$$
and hence
$$P(c,d)Q(c,d)d=P^{\delta}(c,d)Q(c,d)-Q^{\delta}(c,d)P(c,d)$$
This is an algebraic relation over $L$ holding for the generic point $(c,d)$, and hence holding identically.
So
$$P(x,y)Q(x,y)y=P^{\delta}(x,y)Q(x,y)-Q^{\delta}(x,y)P(x,y)$$
The degree of the right-hand-side in $y$ is bounded by $\deg_yP+\deg_yQ$, whereas on the left-hand-side we have $\deg_yP+\deg_yQ+1$.
This contradiction proves the lemma.
\end{proof}

Now we turn to the construction of the $\delta$-variety $X$ whose existence is asserted in Proposition~\ref{exampleprop}.
Let $X\subset K^2$ be the $\delta$-variety defined by the equations
\begin{eqnarray*}
\delta x&=&x^2-y^2\\
\delta y&=&x(x-y)
\end{eqnarray*}
and let $Z:=X(\C)=\{(c,c):c\in \C\}$.
Note that $X$ is irreducible and finite dimensional as it is the set of sharp points of the algebraic $D$-variety $(\mathbb A^2,s)$ where $s(x,y):=(x,y,x^2-y^2,x^2-xy)$.
We want to show that $(T_{\delta}X)|_Z$ is not $\C$-algebraic.

An easy computation using the formulae given in Chapter~VIII~$\S2$ of Kolchin~\cite{kolchin85} for the differential tangent bundle shows that $T_\delta X\subset K^4$ is given by
\begin{eqnarray*}
\delta x&=&x^2-y^2\\
\delta y&=&x(x-y)\\
\delta u&=& 2(xu-yv)\\
\delta v &=& 2xu-yu-xv
\end{eqnarray*}
Restricting to $Z$ we have that $(T_{\delta}X)|_Z\subset K^4$ is defined by
\begin{eqnarray*}
x&=&y\\
\delta x&=&0\\
\delta u&=& 2x(u-v)\\
\delta v &=& x(u-v)
\end{eqnarray*}

\begin{claim}
\label{noninternalimage}
There exists a partial $0$-definable function on $(T_{\delta}X)|_Z$ whose image is the group $G$ from Lemma~\ref{log}.
\end{claim}
\begin{proof}[Proof of Claim~\ref{noninternalimage}]
The equations for $(T_{\delta}X)|_Z$ given above imply that for $u\neq v$, $\displaystyle \delta\left(\frac{\delta (u-v)}{u-v}\right)=0$.
Hence, if we set $W:=(T_{\delta}X)|_Z\cap\{(x,y,u,v):u\neq v\}$, and $f:K^4\to K$ to be the $0$-definable function $(x,y,u,v)\mapsto u-v$, then $f(W)\subseteq G$.

On the other hand, to see that $f$ maps onto $G$, suppose $g\in G$ and let $a\in\C$ be such that $\frac{\delta g}{g}=a$.
Then it is easy to check that $(a,a,2g,g)\in (T_{\delta}X)|_Z$ and $f(a,a,2g,g)=g$.
\end{proof}

Since $G$ is not $\C$-algebraic by Lemma~\ref{log}, the Claim implies that $(T_{\delta}X)|_Z$ is not $\C$-algebraic either.
This completes the proof of Proposition~\ref{exampleprop}.
\qed

\begin{remark}
It may be worth pointing out that Proposition~\ref{exampleprop} gives also an example of another phenomenon that may be of independent interest:  {\em There exists a finite dimensional $\C$-linear space $U\to V$, with $V\subset\C^n$, but such that $U$ is not $\C$-algebraic.}
Here, given a $\Delta$-variety $V$, by a ``$\C$-linear space over $V$" we mean the relative notion of a $\C$-vector space; that is, a surjective $\Delta$-morphism $U\to V$, equipped with $\Delta$-morphisms $+,\lambda,z$
where
\begin{itemize}
\item
the following diagram commutes
$$\xymatrix{
U\times_VU\ar[rr]^+\ar[dr]&& U\ar[dl]\\
&V
}$$
\item
the following diagram commutes
$$\xymatrix{
\C\times U\ar[rr]^{\lambda}\ar[dr]&& U\ar[dl]\\
&V
}$$
\item
$z:V\to U$ is  a section to $U\to V$;
\end{itemize}
such that for all $a\in V$, the fibre $U_a$ is a $\C$-vector space with addition $+_a$, zero $z(a)$ and scalar multiplication $\lambda_a$.
\end{remark}

\begin{proof}
Let $X$ be as in Proposition~\ref{exampleprop}, set $V:=X(\C)$ and $U:=(T_{\Delta}X)|_{X(\C)}$.
Then $U$ is a non-$\C$-algebraic finite dimensional $\C$-linear space over $V\subset\C^n$.
\end{proof}


\end{document}